\documentclass[12pt]{article}
\usepackage{graphicx}
\usepackage{amsmath}
\usepackage{amsfonts}
\usepackage{amsthm}
\def\C{\mathbb{C}}

\def\F{\mathbb{F}}

\newtheorem{theorem}{Theorem}[section]
\newtheorem{lemma}[theorem]{Lemma}
\newtheorem{corollary}[theorem]{Corollary}

\begin{document}

\title{Clifford Algebras and Graphs}

\author{Tanya Khovanova \\
\textit{Department of Mathematics, MIT}}
\date{October 18, 2008}
\maketitle

\begin{abstract}
I show how to associate a Clifford algebra to a graph. I describe the structure of these Clifford graph algebras and provide many examples and pictures. I describe which graphs correspond to isomorphic Clifford algebras and also discuss other related sets of graphs. This construction can be used to build models of representations of simply-laced compact Lie groups.
\end{abstract}


\section{Clifford Algebras}\label{ca}

Let $A$ be a unital algebra over $\C$, with $n$ generators $e_1,e_2,\ldots,e_n$ and relations $e_i^2 = -1$ for any $i$, and $e_ie_j = -e_je_i$, for $i \ne j$. $A$ is a classical Clifford algebra.

As a vector space $A$ has dimension $2^n$ and is generated by monomials $e_{i_1}e_{i_2}\ldots e_{i_k}$, where $i_1 < i_2 < \ldots < i_k$. The monomials are in one-to-one correspondence with the subsets of the set $\{1, 2, \ldots n\}$ or with binary strings of length $n$. 

Suppose $\alpha$ is a binary string of length $n$. We associate with this string the monomial $e_\alpha = e_{i_1}e_{i_2}\ldots e_{i_k}$, where $1 \le i_1 < i_2 < \ldots < i_k \le n$ and $i_1$, $i_2$, $\ldots$, $i_k$ are positons of ones in the string $\alpha$. We associate 1 with the string of all zeroes. If $\beta$ is a binary string too, then $e_\alpha e_\beta = \pm e_\gamma$, where $\gamma = \alpha \text{XOR} \beta$, and XOR is the standard parity (xoring) operation on binary strings.

Let us look at the center of this algebra --- the subalgebra of elements that commute with all elements. Are there central elements not in $\C1$? Every monomial either commutes or anticommutes with generators $e_i$. From this we can deduce that the center is spanned by monomials. Suppose a monomial $c$ in the center is of length $m$; that is, $c$ is the product of $m$ different generators. Then, $e_i c = (-1)^m c e_i$ if $e_i$ is not in the monomial $c$, and $e_i c = (-1)^{m-1} c e_i$ otherwise. From here, we see that $c$ either contains all the generators or none. The product of all the generators $e_1e_2\ldots e_n$ is in the center iff $n$ is odd.

We showed that the classical Clifford algebra has a one-dimensional center for even $n$ and a two-dimensional center for odd $n$.

\section{Clifford Graph Algebras}

Let $G$ be a graph with $n$ vertices and no multiple edges and no loops. We associate with this graph a unital algebra $A_G$ over $\C$, with $n$ generators $e_1$, $e_2$, $\ldots$, $e_n$ corresponding to the vertices; relations $e_i^2 = -1$ for any $i$; and relations $e_ie_j = - e_je_i$, if there is an edge between the $i^{th}$ and $j^{th}$ vertices, and $e_ie_j = e_je_i$, if there is no edge between the $i^{th}$ and $j^{th}$ vertices.

We call the Clifford algebra associated with a given graph a Clifford graph algebra.

The classical Clifford algebra in Section \ref{ca} is a Clifford graph algebra of the complete graph with $n$ vertices. If our graph doesn't have any edges, then its Clifford graph algebra is commutative.

\section{The Center of a Clifford Graph Algebra}

As in the case of the classical Clifford algebra, a Clifford graph algebra has dimension $2^n$ over $\C$, and has a basis of monomials $e_\alpha$, over all binary strings $\alpha$ of length $n$, or, correspondingly subsets of the set \{1, 2, \ldots, n\}.

I would like to describe the center of a Clifford graph algebra. Often, understanding the center of a ring is a first step towards understanding its structure and the structre of its representations. In case of Clifford graph algebras the dimension of the center uniquely determines the structure of the given Clifford algebra as proven in theorem \ref{ccca}. I will denote the center of algebra $A$ by $Z(A)$.

Suppose an element $c = \sum a_\alpha e_\alpha$ is in the center, where $e_\alpha$ are  monomials. As each monomial either commutes or anticommutes with the basis elements $e_i$, every monomial $e_\alpha$ lies in the center. Hence, to analyze the center it is enough to analyze central monomials. A monomial $e_\alpha$ is central iff for each vertex $i$, the number of edges connecting it to the set of vertices $\alpha$ is even.

Let me describe monomials in the center in terms of the adjacency matrix. Let me remind you that the adjacency matrix of a graph $G$ is the matrix $(a_{ij})_{i,j=1}^n$ such that $a_{ij} = 1$ if the vertices $i$ and $j$ are connected and $a_{ij} = 0$ if they are not connected. A central monomial $e_\alpha$ corresponds to a subset of vertices $\alpha$. If we add up rows of the adjacency matrix corresponding to this subset, the resulting row vector will have only even entries. That is, this vector is the zero vector modulo 2. 

The elements of the center form a subalgebra, and the dimension of the center is always a power of 2.

\section{Examples}\label{e}

Let us consider path graphs and star graphs as examples. Not only providing examples can help my readers to understand my construction, but I also have other reasons for providing these particular examples, which I will reveal at the end of this section.

Let us start with a path graph $P$ on $n$ vertices. Vertices $i$ and $j$ are connected by an edge iff $|i-j| = 1$ (see Figure \ref{path6}). 

\begin{figure}[htp]
  \centering
  \includegraphics[scale=0.5]{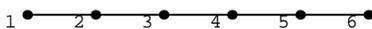}
  \caption{Path Graph on 6 Vertices.}\label{path6}
\end{figure}

If $n$ is even, the center of the corresponding Clifford graph algebra $A_P$ is one-dimensional: $Z(A_P) \simeq \C$. If $n$ is odd, then the center also contains the monomial $e_1e_3\ldots e_n$ --- the product of odd-numbered generators. In Figure \ref{path7} you can see a path graph with 7 vertices. The monomial $e_1e_3e_5e_7$ is in the center of this Clifford graph algebra.

\begin{figure}[htp]
  \centering
  \includegraphics[scale=0.5]{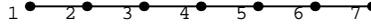}
  \caption{Path Graph on 7 Vertices.}\label{path7}
\end{figure}

Let us compare the Clifford algebras of the complete graph and the path graph with $n$ vertices. If we denote the generators of the Clifford algebra of the complete graph as $e_i$ and the generators of the Clifford algebra of the path graph as $e_i^\prime$, we can show that the Clifford algebras of these two graphs are isomorphic by presenting an isomorphism: 
$$e_1^\prime = e_1 \text{, and }  e_i^\prime = e_{i-1}e_i, \text{ for } i \ne 1,$$
and the other way: 
$$e_i = e_i^\prime e_{i-1}^\prime \cdots e_1^\prime.$$

Let us consider the star graph on $n$ vertices (see Figure \ref{star8}). 

\begin{figure}[htp]
  \centering
  \includegraphics[scale=0.5]{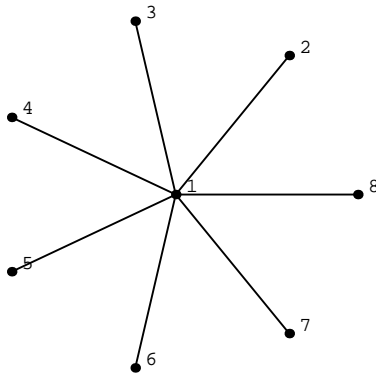}
  \caption{Star Graph on 8 Vertices.}\label{star8}
\end{figure}

In this case, the vertex number 1 connects to every other vertex, and there are no other edges. The center is a linear combination of monomials of even length such that they do not contain $e_1$. The dimension of the center is $2^{n-2}$. Let us show that the Clifford algebra of the star graph is isomorphic to the Clifford algebra of a graph with $n$ vertices and one edge. We denote the generators of the Clifford algebra of the star graph as $e_i$ and the generators of the Clifford algebra of the graph with $n$ vertices and one edge as $e_i^\prime$, where we number the vertices adjacent to the edge as 1 and 2, so that $e_1^\prime e_2^\prime = -e_2^\prime e_1^\prime$, and all other pairs of generators commute. We can show that the Clifford algebras of these two graphs are isomorphic by presenting an isomorphism: 
$$e_1^\prime = e_1 \text{, } e_2^\prime = e_1 e_2 \text{ and for } i > 2 \text{, } e_i^\prime = e_2 e_i,$$ and the other way: 
$$e_1 = e_1^\prime \text{, } e_2 = e_1^\prime e_2^\prime \text{ and } e_i = e_i^\prime e_1^\prime e_2^\prime \text{ for } i > 2.$$

Now it is time to explain why I particularly like these two examples. The path graph and the star graph with $n$ vertices are quite similar: they have the same number of edges --- $n-1$, not to mention that for $n$ equal 2 or 3 the path and the star graphs are isomorphic. On the other hand, the Clifford algebra of a path graph has the smallest possible center (we will see the proof later) and the Clifford algebra of a star graph has largest center for a non-commutative Clifford graph algebra of this size.

\section{Clifford Graph Algebra Structure}\label{cgas}

Every central monomial in a Clifford graph algebra corresponds to a central idempotent. Suppose $e_\alpha$ is a central monomial. Then, $e_\alpha^2 = \pm 1$. Let us denote by $f_\alpha$ a multiple of $e_\alpha$ such that $f_\alpha^2 = 1$. If $e_\alpha^2 = 1$, then $f_\alpha = e_\alpha$ and if $e_\alpha^2 = -1$ then $f_\alpha = \imath e_\alpha$. The element $c = (1+f_\alpha)/2$ is a central idempotent. Being a central idempotent means that $c$ is in the center and $c^2 = c$, which is easy to check. Central idempotent $c$ gives rise to a decomposition of our Clifford graph algebra $A$ as a direct sum of $Ac$ and $A(1-c)$ (see Idempotence article at wiki \cite{wiki_Idempotence}). 

Let us take an index $j$ that belongs to the set $\alpha$, that is, $e_j$ is one of the generators that the monomial $e_\alpha$ contains. Denote the graph $G$ with the vertex corresponding to $j$ removed by $G_{-j}$. The algebra $A_{G_{-j}}$ is naturally embedded into $A_G$. The map $x \to xc$ is an isomorphism between this embedding of $A_{G_{-j}}$ and $Ac$. Similarly, the map $x \to x(1-c)$ is an isomorphism between the embedding of $A_{G_{-j}}$ and $A(1-c)$. Hence, $A_G$ is a direct sum of two copies of $A_{G_{-j}}$. We can continue the decomposition until we get a graph such that its Clifford algebra has a one-dimensional center. 

If the dimension of the center of $A_G$ is $2^k$, then $A_G$ is a direct sum of $2^k$ copies of the Clifford graph algebra of a subgraph $G^\prime$ of $G$, such that $A_{G^\prime}$ has a one-dimensional center.

By the Artin-Wedderburn theorem \cite{Artin} we can show that a Clifford graph algebra is isomorphic to a direct sum of matrix algebras. Therefore, any Clifford graph algebra is a direct sum of $2^k$ copies of an $m \times m$ matrix algebra, for some $k$. We can deduce from here that $m = 2^{(n-k)/2}$. In particular, we see that for graphs with odd number of vertices, the corresponding Clifford graph algebra has to have a center of dimension more than 1.

In our examples, we saw that the Clifford graph algebra of a complete graph is isomorphic to a Clifford algebra of a path graph. For $n$ even, both of these algebras are isomorphic to a matrix algebra over the $2^{n/2}$-dimensional space. For odd $n$, the Clifford algebra of the complete graph is a direct sum of two Clifford graph algebras of the complete graph for $n-1$. The same goes for the path graph.

The Clifford algebra corresponding to a star graph is the direct sum of $2^{n-2}$ copies of a $2 \times 2$ matrix algebra.

\section{Small Graphs}

Previously we covered complete graphs, path graphs and star graphs in our examples. Now I would like to cover all graphs with small number of vertices and see how their Clifford graph algebras are different from each other.

Let us denote by $\text{Cliff}_k(n)$ the number of graphs with $n$ vertices such that the center of their Clifford graph algebras has dimension $2^k$.

For $n=1$, there is only one graph and its algebra is $\C \oplus \C$. $\text{Cliff}_2(1) = 1$.

For $n=2$, there are two graphs --- a complete graph whose Clifford algebra is a $2 \times 2$ matrix algebra $\text{Mat}(2)$ and the dual graph without edges whose Clifford algebra is $\C \oplus \C \oplus \C \oplus \C$. Thus, $\text{Cliff}_1(2) = 1$ and $\text{Cliff}_4(2) = 1$.

For $n=3$, there are four graphs. The Clifford algebra of a graph with no edges is commutative, thus it has an 8-dimensional center and is equal to the direct sum of 8 copies of $\C$. All other graphs (see Figure \ref{ThreeGraphs}) have a two-dimensional center, hence they are isomorphic to $\text{Mat}(2) \oplus \text{Mat}(2)$. Thus, $\text{Cliff}_2(3) = 3$ and $\text{Cliff}_8(3) = 1$.

\begin{figure}[htp]
  \centering
  \includegraphics[scale=0.5]{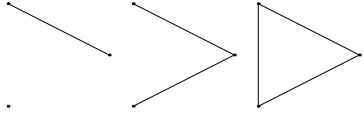}
  \caption{Graphs with 3 vertices and noncommutative Clifford algebras.}\label{ThreeGraphs}
\end{figure}

For $n=4$, there are 11 graphs. Below I describe all the graphs with 4 vertices and present the dimension of the center of the corresponding Clifford algebra.

\begin{enumerate}
\item There is one graph without edges --- the dimension is 16. 
\item There is one graph with one edge --- the dimension is 4. 
\item There are two graphs with two edges: the graph with two adjacent edges --- the dimension is 4; the graph with 2 non-adjacent edges --- the dimension is 1. 
\item There are 3 graphs with 3 edges: the complete graph with 3 vertices plus 1 isolated vertex --- the dimension is 4; the path graph --- the dimension is 1; the star graph --- the dimension is 4. 
\item There are two graphs with 4 edges: the cycle graph --- the dimension is 4; the kite graph which is the dual graph to the graph with two edges adjacent to each other --- the dimension is 1.
\item There is one graph with 5 edges --- the dimension is 4.
\item The only graph with 6 edges is the complete graph --- the dimension is \nolinebreak[1] 1.
\end{enumerate}

You can see graphs with 4 vertices and their corresponding dimensions in Figure \ref{FourGraphs}.  As we calculated, $\text{Cliff}_1(4) = 4$, $\text{Cliff}_4(4) = 6$, and  $\text{Cliff}_{16}(4) = 1$.

\begin{figure}[htp]
  \centering
  \includegraphics[scale=0.5]{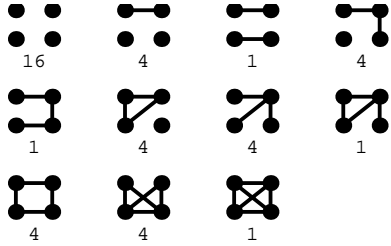}
  \caption{Graphs with four vertices.}\label{FourGraphs}
\end{figure}

\section{Clifford Class}\label{cc}

Let us say that two graphs with the same number of vertices belong to the same Clifford class if the centers of their Clifford algebras have the same dimension. 

\begin{theorem}\label{ccca}
If two graphs are in the same Clifford class then their Clifford algebras are isomorphic.
\end{theorem}

\begin{proof}
The structure of a Clifford graph algebra is uniquely defined by the number of vertices and the dimension of its center.
\end{proof}

Let us consider a special graph with $2k + m$ vertices and $k$ edges which has $m$ isolated vertices and $2k$ vertices of degree one. I denote this graph as $G(k,m)$. This graph is the union of $k$ $K_2$-graphs and $m$ $K_1$-graphs. (Here I use the standard definition, $K_n$, for a complete graph with $n$ vertices.) It is easy to check that the center of the Clifford algebra of this graph has dimension $2^m$ and the algebra itself is isomorphic to the direct sum of $2^m$ matrix algebras of size $2^k \times 2^k$.

\begin{lemma}\label{Constr} 
Suppose $G$ is a graph with at least one edge, then there is a graph $G^\prime$ such that $G^\prime$ is a union of $K_2$ and some graph and $A_G$ is isomorphic to $A_{G^\prime}$.
\end{lemma}

\begin{proof}
Let us number the vertices in graph $G$ in such a way that the vertices numbered 1 and 2 are connected. Suppose $e_i$ are the generators in $A_G$ --- the Clifford algebra of the graph $G$. We will build new generators in the algebra $A_G$. We will have $e_1^\prime = e_1$ and $e_2^\prime = e_2$. For every $i > 2$ let us choose a new generator $e_i^\prime$ in the following way:

\begin{itemize}
\item If the vertex $i$ is not connected to either vertex 1 or vertex 2, then $e_i^\prime = e_i$
\item If the vertex $i$ is connected to vertex 1 and is not connected to vertex 2, then $e_i^\prime = \imath e_i e_2$
\item If the vertex $i$ is not connected to vertex 1 and is connected to vertex 2, then $e_i^\prime = \imath e_i e_1$
\item If the vertex $i$ is connected to both vertices 1 and 2, then $e_i^\prime = e_i e_2 e_1$
\end{itemize}

The new generators $e_i^\prime$ have the property ${e_i^\prime}^2 = -1$ and $e_ie_j = \pm e_je_i$. Hence, we can build a graph $G^\prime$ for which the generators $e_i^\prime$ generate its Clifford graph algebra.

The algebra $A_{G^\prime}$ is isomorphic to $A_G$. The generators $e_i^\prime$ commute with both $e_1^\prime$ and $e_2^\prime$ for $i > 2$. This means that vertices numbered 1 and 2 in the graph $G^\prime$ are isolated from the rest of the graph $G^\prime$. From here we see that $G^\prime$ is the union of $K_2$ and another graph.
\end{proof}

\begin{theorem}\label{ClCl} 
Each Clifford class has exactly one representative of type $G(k,m)$.
\end{theorem}

\begin{proof}
The theorem is trivial for graphs with one or two vertices. If a graph with $n$ vertices doesn't have an edge, then the graph is $G(0,n)$. If a graph $G$ has an edge, then we can build a graph $G^\prime$ which is a union of $K_2$ and a graph with $n-2$ vertices and such that $A_G$ is isomorphic to $A_{G^\prime}$. We can use induction on the number of vertices to finish the proof.
\end{proof}

\begin{corollary}\label{cgastructure}
A Clifford algebra of a graph is isomorphic to a direct sum of $2^m$ copies of the matrix algebra over $2^k$-dimensional space for some $k$ and $m$ such that $n = 2k+m$.
\end{corollary}

\begin{proof}
We already proved this fact in Section \ref{cgas} using the Artin-Wedderburn theorem. Here we prove it again without using this theorem. From theorem \ref{ClCl}, all Clifford algebras with one-dimensional center and the same number of vertices are isomorphic to each other. That means that the Clifford algebra of the graph $G(k, 0)$, sometimes called the ladder rung graph, is isomorphic to the Clifford graph algebra of the complete graph $K_{2k}$, which is the matrix algebra $\text{Mat}_{2^k}$. Hence the Clifford algebra of $G(k,m)$, as well as all other algebras of the same class, is isomorphic to a direct sum of $2^m$ copies of $\text{Mat}_{2^k}$.
\end{proof}

\section{Adjacency Matrices}

The adjacency matrix of a graph $G$ is a matrix $(a_{ij})_{i,j=1}^n$ such that $a_{ij} = 1$ if the vertices $i$ and $j$ are connected and $a_{ij} = 0$ if they are not connected. By definition, the adjacency matrix is symmetric: $a_{ij} = a_{ji}$; and by requiring our graphs not to have loops we restrict the diagonal to have only zeroes: $a_{ii} = 0$.

In the proof of lemma \ref{Constr} we used a construction where we express old generators through new generators. We can express our replacement through a sequence of basic replacements. The basic replacement is the following. The new generators are $e_i^\prime = e_i$ for $i \ne 2$ and $e_2^\prime = a e_1 e_2$. The coefficient $a$ is needed to adjust the square of $e_2^\prime$ to be equal to $-1$. The new adjacency matrix $a_{ij}^\prime$ can be calculated from the old adjacency matrix by the following operation: replace the second row in the adjacency matrix by the sum of the first row and the second row modulo 2, then do the same operation on the columns.

In particular, we see that the basic operation doesn't change the rank or the determinant of the adjacency matrix if we consider this matrix to be a matrix over field $\F_2$. In particular, the basic operation doesn't change the parity of the determinant.

\begin{lemma}
Two graphs belong to the same Clifford class iff their adjacency matrices considered as matrices over $\F_2$ have the same rank.
\end{lemma}

\begin{proof}
I previously produced a construction that changes a graph but doesn't change its Clifford class. I just showed that this construction doesn't change the rank of the adjacency matrix over $\F_2$. I also showed that any graph has exactly one graph of type $G(k,m)$ in its class. The adjacency matrix of a Clifford algebra of the graph $G(k,m)$ has rank $2k$. Thus, the rank separates the classes.
\end{proof}

\begin{corollary}
If the rank of the adjacency matrix considered over $\F_2$ of a graph is equal to 2k, then the center of the Clifford algebra corresponding to this graph is $2^{n-2k}$, where $n$ is the number of vertices.
\end{corollary}

\section{Odd-Determinant Graphs}

Let us call a graph an odd-determinant graph if its Clifford graph algebra has a one-dimensional center. 

\begin{theorem}
The odd-determinant graphs are the graphs whose adjacency matrix has an odd determinant.
\end{theorem}

\begin{proof}
By the discussion above, the parity of the determinant of the adjacency matrix is invariant under basic construction, and by using the basic construction we can replace a given graph $G$ with a canonical graph $G(k,0)$ such that their Clifford algebras are isomorphic.
\end{proof}

Obviously, an odd-determinant graph doesn't have isolated vertices. Also, all odd-determinant graphs are in the same Clifford class. No graphs in the Clifford class of an odd-determinant graph have isolated vertices. And vice versa, if no graphs in the Clifford class of a graph have isolated vertices then all the graphs in this class are odd-determinant graphs. 

From the classification theorem it follows that an odd-determinant graph has a complete graph on even number of vertices in its class.

In the next two sections I discuss other sets of graphs that are related to odd-determinant graphs.

\section{Mating Graphs}

By definition, a mating graph, sometimes called a point-determining graph (see \cite{Gessel}), is a graph such that no two vertices have identical sets of neighbors.

\begin{lemma}
Odd-determinant graphs are mating graphs.
\end{lemma}

\begin{proof}
If two vertices $i$ and $j$ of the graph $G$ have the same set of neighbors, then the product of the corresponding generators $e_ie_j$ is in the center of the Clifford graph algebra $A_G$. Hence, odd-determinant graphs can't have two vertices with the same set of neighbors.
\end{proof}

The converse of the lemma is not true. There are mating graphs that are not odd-determinant graphs. The smallest examples have three vertices. Graphs with an odd number of vertices can't be odd-determinant. At the same time there are 2 mating graphs with 3 vertices:  $K_3$, and the union of $K_2$ and an isolated point (See Figure \ref{MatingGraphs3}).  

\begin{figure}[htp]
  \centering
  \includegraphics[scale=0.5]{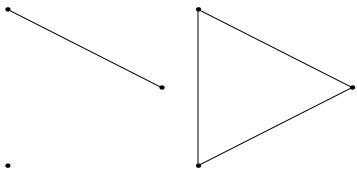}
  \caption{Mating graphs with 3 vertices.}\label{MatingGraphs3}
\end{figure}

It is easy to see that complete graphs are always mating graphs. The set of neighbors of each vertex is uniquely determined by the vertex itself: the neighbors of the vertex are the set of all the other vertices. That means the complete graphs with an odd number of vertices give an infinite set of examples of mating and not odd-determinant graphs. Similarly, the union of a complete graph with an isolated point is a mating graph.

Below I present all mating graphs with 4 vertices. Only one of them is not an odd-determinant graph --- the union of the complete graph $K_3$ and an isolated vertex.

\begin{figure}[htp]
  \centering
  \includegraphics[scale=0.5]{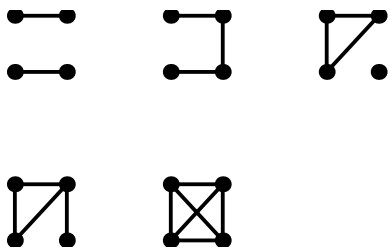}
  \caption{Mating graphs with 4 vertices.}\label{MatingGraphs4}
\end{figure}

We see that if we want to estimate the number of odd-determinant graphs we have a natural bound: the number of such graphs is not more than the number of mating graphs.

There is a natural set of graphs that is somewhat in between odd-de\-ter\-mi\-nant graphs and mating graphs. This is the set of graphs with invertible adjacency matrices.

\section{Invertible Adjacency Matrix Graphs}

\begin{lemma}
Odd-determinant graphs have an invertible adjacency matrix. Graphs that have an invertible adjacency matrix are mating graphs.
\end{lemma}

\begin{proof}
Graphs that have an invertible adjacency matrix have a nonzero determinant. Odd-determinant graphs are a subset of them. On the other hand a non-mating graph has two equal rows in its adjacency matrix, thus the determinant of its adjacency matrix is zero.
\end{proof}

Thus, graphs with an invertible adjacency matrix give us a better bound on the number of odd-determinant graphs than mating graphs.

The smallest graph with an invertible adjacency matrix that has an even determinant is $K_3$. Any graph with an invertible adjacency matrix and an odd number of vertices is not an odd-determinant graph. Let us find an example of a graph with an even number of vertices and an invertible adjacency matrix which is not an odd-determinant graph. We can see that there are no such graphs with 2 or 4 vertices. Hence, we should try six vertices. There are 10 such graphs with 6 vertices and you can see them in Figure \ref{InvertibleEven6}. Not surprisingly, the union of two copies of $K_3$ is in this set.

\begin{figure}[htp]
  \centering
  \includegraphics[scale=0.5]{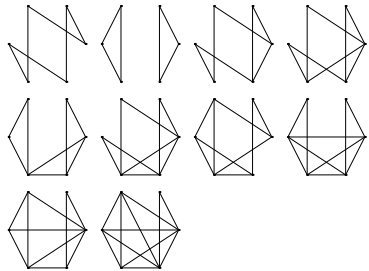}
  \caption{Graphs with an invertible adjacency matrix with an even determinant.}\label{InvertibleEven6}
\end{figure}

To complete our discussion I would like to show mating graphs with a degenerate adjacency matrix. The smallest such graph has one vertex. The smallest nontrivial case is a graph with 3 vertices and it is the union of $K_2$ and $K_1$.

\section{Unions of Graphs}

I already mentioned the unions of graphs, but this whole section  is dedicated to the unions of graphs, so I would like to remind you of the formal definition.

Suppose we are given two graphs $G_1$ and $G_2$ with disjoint sets of vertices $V_1$ and $V_2$ and the sets of edges $X_1$ and $X_2$. The union graph $G$ has $V$ as its set of vertices and $X$ as its set of edges, where $V$ is the union of $V_1$ and $V_2$ and $X$ is the union of $X_1$ and $X_2$.

Let us see how mating graphs behave with respect to unions. Obviously, if a graph contains at least two isolated vertices, it can't be a mating graph. At the same time we can build many mating graphs by taking unions. The following lemma is easy to prove.

\begin{lemma}
A union of mating graphs is a mating graph if it doesn't contain more than one isolated vertex.
\end{lemma}

This lemma allows us to provide many more examples of mating graphs. In particular, unions of complete graphs as long as they do not contain more than one isolated vertex are mating graphs.

The determinants of adjacency matrices behave nicely with respect to unions. If graph $G$ is the union of $G_1$ and $G_2$, then the determinant of the adjacency matrix of $G$ is the product of the corresponding determinants for $G_1$ and $G_2$. From here the next lemma follows:

\begin{lemma}
If $G_1$ and $G_2$ are two graphs with invertible adjacency matrices, then their union $G$ has an invertible adjacency matrix. If the union $G$ of two graph $G_1$ and $G_2$ has an invertible adjacency matrix, then the graphs $G_1$ and $G_2$ themselves have invertible adjacency matrices.
\end{lemma}

The following two lemmas can be trivially proved using the determinant argument, but I would like to use Clifford graph algebras to prove them. The reason I am doing this is that I would like to share with you the beauty of Clifford graph algebras.

\begin{lemma}
If $G_1$ and $G_2$ are two odd-determinant graphs, then the union graph $G$ is an odd-determinant graph.
\end{lemma}

\begin{proof}
Suppose a monomial $m$ is in the center of the Clifford algebra $A_G$. Then we can express it as a product of two monomials $m_1$ and $m_2$, where $m_1 (\text{or }m_2)$ contains only the generators corresponding to the first (or second) graph. It is easy to see that monomials $m_i$ must belong to the center of the Clifford graph algebra $A_{G_i}$.
\end{proof}

The converse is also true:

\begin{lemma}
If $G$ is an odd-determinant graph and is the union of several connected components, then each component is an odd-determinant graph too.
\end{lemma}

\begin{proof}
If a Clifford algebra corresponding to a connected component has a center, then the corresponding monomial is in the center of the Clifford algebra of the whole graph.
\end{proof}

The previous lemmas allow us to build many mating graphs that are not odd-determinant graphs. For example, the union of several complete graphs is a mating and at the same time not an odd-determinant graph if it contains not more than one isolated vertex and if at least one of the complete graphs in the union has an odd number of vertices.

\section{Sequences}

For the reference, I would like to present here the list of sequences related to this paper. These are the sequences of different types of graphs indexed by the number of vertices. 

The shortest sequence in my list is the sequence of odd-determinant graphs. Obviously, odd-determinant graphs can have only an even number of vertices. That means, the sequence I am talking about is the sequence $a(n)$ of odd-determinant graphs with $2n$ vertices. As I showed before the sequence starts as: $a(1) = 1$, $a(2) = 4$. I checked that there are 47 odd-determinant graphs of order 6. (see Figures \ref{ODGraphs1}, \ref{ODGraphs2}, \ref{ODGraphs3}, \ref{ODGraphs4}) This means that $a(3) = 47$.

\begin{figure}[htp]
  \centering
  \includegraphics[scale=0.5]{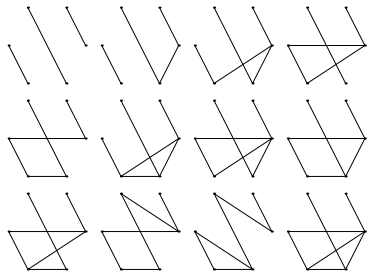}
  \caption{Odd-determinant graphs with 6 vertices. Part 1.}\label{ODGraphs1}
\end{figure}

\begin{figure}[htp]
  \centering
  \includegraphics[scale=0.5]{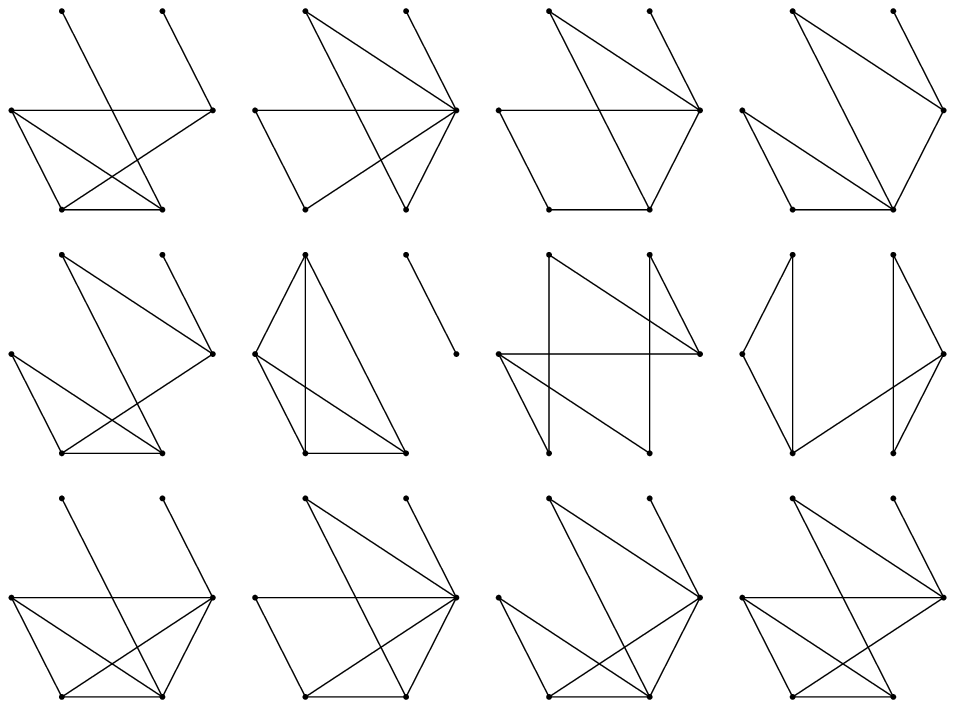}
  \caption{Odd-determinant graphs with 6 vertices. Part 2.}\label{ODGraphs2}
\end{figure}

\begin{figure}[htp]
  \centering
  \includegraphics[scale=0.5]{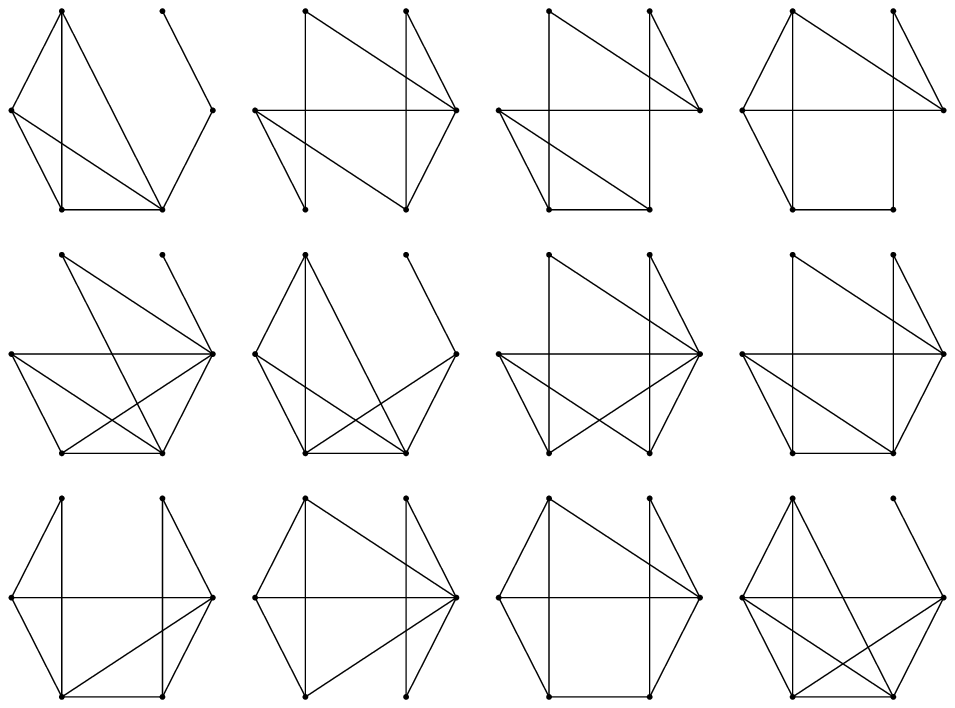}
  \caption{Odd-determinant graphs with 6 vertices. Part 3.}\label{ODGraphs3}
\end{figure}

\begin{figure}[htp]
  \centering
  \includegraphics[scale=0.5]{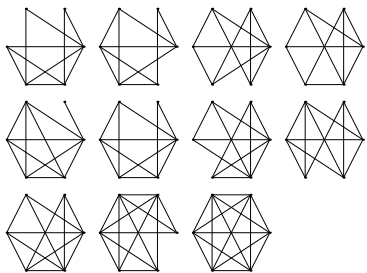}
  \caption{Odd-determinant graphs with 6 vertices. Part 4.}\label{ODGraphs4}
\end{figure}

I submitted this sequence to the Online Encyclopedia of Integer Sequences, see \cite{OEIS} and its number in the OEIS is A141040. I will appreciate if someone can expand it beyond the first three terms. 

\begin{itemize}
\item Number of odd-determinant graphs with $2n$ vertices (A141040): 1, 4, 47, $\ldots$.
\end{itemize}

In the list below I show other sequences related to this paper. The $n$-th term of each sequence is the number of particular graphs with $n$ vertices. All sequences start with index 1, that is, the first term is the number of particular graphs with one vertex. The A number of the sequence references the sequence in the Online Encyclopedia of Integer Sequences, see \cite{OEIS}.

\begin{itemize}
\item Number of graphs with $n$ unlabeled vertices (A000088): 1, 2, 4, 11, 34, 156, 1044, 12346, $\ldots$.

\item Number of even-determinant graphs with $n$ vertices (A140981): 1, 1, 4, 7, 34, 109, 1044, $\ldots$. This sequence is the difference between the sequence of all graphs and the sequence of odd-determinant graphs.

\item Number of mating graphs with $n$ vertices (A004110): 1, 1, 2, 5, 16, 78, 588, $\ldots$. This sequence is the same as sequence of $n$-node graphs without endpoints as proved in \cite{Kilibarda}.

\item Number of non-mating graphs with $n$ vertices (A141580): 0, 1, 2, 6, 18, 78, 456, $\ldots$. This sequence is the difference between the sequence of all graphs and the sequence of mating graphs.

\item Number of graphs with $n$ vertices and an invertible adjacency matrix (A109717): 0, 1, 1, 4, 9, 57, 354, $\ldots$.

\item Number of graphs with $n$ vertices and a degenerate adjacency matrix (A133206): 1, 1, 3, 7, 25, 99, 690, $\ldots$. This sequence is the difference between the sequence of all graphs and the sequence of graphs with an invertible adjacency matrix.

\item Number of mating graphs with $n$ vertices and a degenerate adjacency matrix (A133279): 1, 0, 1, 1, 7, 21, 234,  $\ldots$. This sequence is the difference between the sequence of mating graphs and the sequence of graphs with an invertible adjacency matrix.

\item Number of even-determinant graphs with $n$ vertices and an invertible adjacency matrix (A103869): 0, 0, 1, 0, 9, 10, 354, $\ldots$. This sequence is the difference between the sequence of even-determinant graphs and the sequence of graphs with a degenerate adjacency matrix.
\end{itemize}

\section{Dynkin Diagrams}

We can build Clifford graph algebras corresponding to the Dynkin diagrams of simply-laced Lie algebras. Here is the table of the dimensions of the center for such algebras. Recall that Lie algebras $A_n$ are defined for $n > 0$ and $D_n$ for $n > 3$.

\begin{center}
\begin{tabular}[c]{|c|c|}
\hline
Lie Algebra & Dimension\\
\hline
$A_{2k-1}$   & 2  \\
$A_{2k}$     & 1  \\
$D_{2k-1}$   & 2  \\
$D_{2k}$     & 4  \\
$E_6$        & 1  \\
$E_7$        & 2  \\
$E_8$        & 1  \\
\hline
\end{tabular}
\end{center}

Clifford graph algebras for Dynkin diagrams were used in my masters' thesis to build models of representations. A model of representations of a group is a representation that contains every irreducible representation exactly once. The construction without proofs is presented in \cite{TKh_Models}.

\section{Conclusions}

I have shown how to construct a Clifford graph algebra, an interesting and simple algebraic object corresponding to a graph. It is a direct sum of $2^m$ copies of matrix algebra $\text{Mat}_{2^k}$, where $2k+m$ is the number of vertices. The structure of the Clifford graph algebra is completely defined by the rank, $2k$, of the graph's adjacency matrix modulo 2. 

Graphs whose Clifford algebras have a one-dimensional center are the most interesting, because they represent the least degenerate case. I called these graphs odd-determinant graphs. The odd-determinant graphs always have an invertible adjacency matrix and, consequently, are always mating graphs.

The Clifford algebras corresponding to Dynkin diagrams are very useful in representation theory. Here too, odd-determinant graphs $A_{2k}, E_6$ and $E_8$ produce the least degenerate construction of models of representations.

\section{Acknowledgements}

I am thankful to Mikhail Khovanov for listening to me and supporting me when I was overexcited working on this paper. I also thank Mikhail Khovanov for reading the second draft of this paper. I am grateful to Julie Sussman, PPA, for editing this paper.


\begin{thebibliography}{9}

\bibitem{Gessel}
Ira Gessel, Ji Li, On Point-Determining Graphs, arXiv:0705.0042v1 (2007)

\bibitem{TKh_Models}
T.~G.~Khovanova, Representation models and generalized Clifford algebras, Funkts. Anal. Prilozh., (1982) \textbf{16}:4, 90-91. 

\bibitem{Kilibarda}
Goran Kilibarda, Enumeration of Unlabelled Mating Graphs, Graphs and Combinatorics, (2007) \textbf{23}, 183-199. 

\bibitem{OEIS} 
N. J. A. Sloane, Online Encyclopedia of Integer Sequences (OEIS). \emph{http://www.research.att.com/\-$\sim$njas/\-sequences/}


\bibitem{Artin} 
Wiki article on the Artin-Wedderburn theorem: \emph{http://en.wikipedia.org/\-wiki/\-Artin-Wedderburn\_theorem}

\bibitem{wiki_Idempotence} 
Wiki article on Idempotence: \emph{http://en.wikipedia.org/\-wiki/Idempotent}


\end{thebibliography}
\end{document}